\documentclass{amsart}
\usepackage{ObermanLatexPackage}
\usepackage{enumerate}
\usepackage{cleveref}

\usepackage{siunitx} % For displaying 1e-6 properly
\newcommand{\Td}{{\mathbb{T}^d}}
\newcommand{\T}{{\mathbb{T}}}
\newcommand{\Sd}{{\mathcal{S}^d}}

\newcommand{\dd}{\,\mathrm{d}}
\newcommand{\HM}[0]{\operatorname{HM}}

\begin{document}
\title[Homogenization of Pucci PDE]{Approximate homogenization of fully nonlinear elliptic PDEs: estimates and numerical results for Pucci type equations}
\author{Chris Finlay}
\author{Adam~M. Oberman}
\date{\today}
%\thanks{The authors are thankful to Yifeng Yu for valuable discussions}
\begin{abstract}
  We are interested in the shape of the homogenized operator $\overline F(Q)$ for PDEs which have the structure of a nonlinear Pucci operator. A typical operator is $H^{a_1,a_2}(Q,x) = a_1(x) \lambda_{\min}(Q) + a_2(x)\lambda_{\max}(Q)$.
  Linearization of the operator leads to a non-divergence form homogenization problem, which can be solved by averaging against the invariant measure.  We estimate the error obtained by linearization based on semi-concavity estimates on the nonlinear operator.  These estimates show that away from high curvature regions, the linearization can be accurate.  Numerical results show that for many values of $Q$, the linearization is highly accurate, and that even near corners, the error can be small (a few percent) even for relatively wide ranges of the coefficients.
\end{abstract}

%\subjclass[2000]{Primary: 35J70, 52A41, Secondary:  93E20, 65N06 }
%, 65N06, 65N12, 65M06, 65M12, 35B50, 35J60, 35R35, 35K65, 49L25}
% 35J70 Elliptic partial differential equations of degenerate type
% 26B25 Convexity, generalizations
% 52A41 Convex functions and convex programs
% 52A40 Inequalities and extremum problems
% 65N06 Finite difference methods ( in Partial differential equations, boundary value problems )
% 93E20 Optimal stochastic control
% Variational Problems
% Measures etc
\keywords{Viscosity Solutions, Partial Differential Equations, Homogenization, Pucci Maximal Operator}
%rubber: shell_escape
\maketitle
%\tableofcontents

\section{Introduction}
In this article we consider fully nonlinear, uniformly elliptic PDEs $F(Q,x)$.
We are interested in approximating the homogenized operator
$\overline  F(Q)$.  We focus on Pucci-type PDE operators in two dimensions.  The
restriction to two dimensions is for computational simplicity and also for
visualization purposes.  We consider periodic coefficients, although in our
numerical experiments we obtained very similar results with random coefficients.

The approach we take is to linearize the operator about the value $Q$,  and to homogenize  the linearized operator $\overline L(Q)$.  The solution of the linear homogenization problem can be expressed  (and in some cases solved analytically) by averaging against the invariant measure.   The result is given by
\[
  \overline {L^Q}(Q) = \int F(Q,x) \rho^Q(x) dx
\]
where $\rho^Q$ is the invariant measure of the corresponding linear problem.
%In the special case of  separable coefficients, $F(Q,x) = a(x)F_0(Q)$, there is an analytic formula for the homogenization of the linear operator, $\overline{L^Q}(Q) = \overline{a} F_0(Q)$ where $\overline{a}$ is the harmonic mean.
We estimate the linearization error
\[
  E(Q) \equiv \overline F(Q) - \overline{L^Q}(Q),
\]
For convex operators, the analysis gives a one sided bound on the error.
In general, we obtain upper or lower bounds on the error, which depend on generalized semi-concavity/convexity estimates of $F$, as well as on  the solution of the cell problem $u^Q$ for the nonlinear problem.
These results are stated in \Cref{th:main} below.

For theoretical results on nonlinear homogenization, we refer to the review \cite{Engquist2008} as well as recent
works on rates of convergence (for example \cite{Armstrong2014}).  There are fewer works which aim to
determine the values $\overline F(Q)$.  Few analytical results are available.   Numerical homogenizing results for Pucci type operators can be found in  \cite{caffarelli2008numerical} using a least-squares formulation.
We also mention numerical work by \cite{GO04} and \cite{oberman2009homogenization} and \cite{luo2011new} in the first order case, as well as \cite{Froese2009} in the second order linear non-divergence case.

The typical operator we consider herein is defined next.  Below, we consider more operators, including the usual convex Pucci Maximal operator.

\begin{definition}[Fully nonlinear elliptic operator $F(Q,x)$ and linearization]\label{def:F}
  We are given  $F: \mathcal{S}^d \times \mathbb{T}^d \to \R$  which is uniformly elliptic, Lipschitz continuous in the first variable and bounded in the second variable.
  Suppose for a given $Q$, that $\grad_Q F(Q,x)$ exists for all $x$.  Write
  \begin{equation}
    \label{LFaffine}
    L^Q(M,x) = \grad_Q  F(Q,x)\cdot(M-Q) +F(Q,x)
  \end{equation}
  for the affine approximation to $F$ at $Q$.
\end{definition}

Given $Q \in \mathcal{S}^d$, write, for $d=2$, $\lambda_{\min}(Q)$ and $\lambda_{\max}(Q)$
for the smaller, and larger eigenvalues of $Q$, respectively.
\begin{example}[Typical PDE operator]
  Given $\delta > 0$, and periodic functions $a_1(y), a_2(y) \geq \delta$. Define the homogeneous order one PDE operator
  \begin{equation}
    H^{a_1,a_2}(Q,x) = a_1(x) \lambda_{\min}(Q) + a_2(x)\lambda_{\max}(Q)
    \label{eq:HP}
  \end{equation}
  Suppose $Q$ has unit eigenvectors  $v_1, v_2$ corresponding to the eigenvalues $\lambda_{\min}(Q),\lambda_{\max}(Q)$, respectively.  Then the linearization at $Q$, of $H^{a_1,a_2}$ is given by
  \begin{equation}
    L^Q(M,x) = a_1(x) v_1^T M v_1 + a_2(x) v_2^T M v_2
    \label{eq:LinearHP}
  \end{equation}
\end{example}

\begin{remark}[Typical results]
  We consider the case of coefficients which are either (i) periodic checkerboards
  or (ii) random checkerboards.   We compute both the nonlinear homogenization
  $\overline F(Q)$ and the homogenized linear operator $\overline{L^Q}(Q)$.  In
  practice, the numerically computed error  is insignificant, less than \num{1e-8} for values of $Q$, away from regions of high curvature of $F$ with respect to $Q$.  Areas where the error is significant correspond to regions where the  semi-concavity constants are large. A typical result is displayed in Figure~\ref{fig1}.  The solid line is a level set of the homogenized linear operator $ \overline{L^Q}(Q)$.   The dots are numerical computations of  $\overline F(Q)$.   The error is very small, except at one point, which corresponds to a corner of the operator.  (Dashed lines indicate underlying operators which comprise $F(Q,x)$.

  Our analysis depends on the shape of $F(Q,x)$ in $Q$, but not on the pattern of
  the coefficients in $x$.  We also considered the case of stripe coefficients.
  For separable examples, the linear approximation is still effective.  However,
  we also found nonseparable examples  where the linear approximation is poor,
  which we will address in a companion paper \cite{ObermanFinlay_LP} with a closer
  bound.
\end{remark}

\begin{figure}
  \includegraphics[width=.5\textwidth]{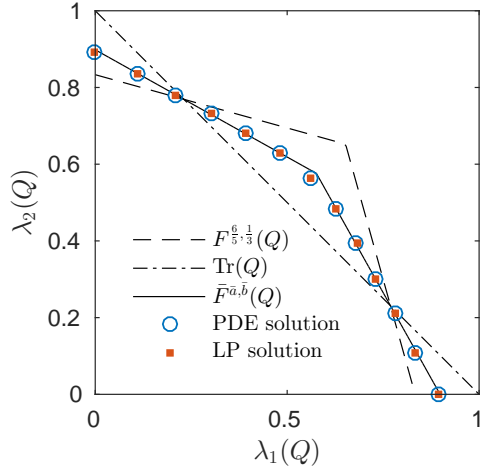}
  \caption{Plot of a single level set of $\overline{L^Q}(Q)$ and $\overline F(Q)$.  This example is typical.  In this case the coefficients are on a random checkerboard. The error is only visible near the corner of the level set of the operator. $F(Q,y)$ a Pucci-type operator, see Definition~\ref{defnPucci} below.  The details of the coefficients can be found in Section~\ref{sec:Numerics}. }
  \label{fig1}
\end{figure}

\subsection{Background: cell problem and linear homogenization}
In this section, we review background material on the cell problem for the nonlinear PDE, and on linear homogenization.  We also give an exact formula in one dimension for a separable operator.

Given $a(y): \T \to \R$, positive,  $a(y)>0$, write  $\HM(a) = \left( \int_\T \frac{\dd y }{a(y)}\right)^{-1}$ for the harmonic mean of $a$.

In the linear case, $\overline L$ can be found by averaging against the
invariant measure, by solving the adjoint equation (see
\cite{bensoussan2011asymptotic} or \cite{Froese2009}), which yields the
following formula.
\begin{lemma}[Linear Homogenization Formula] \label{formula:HM}
  The separable linear operator $L(M,x) = a(x)A_0:M + f(x)$ has invariant
  measure $\rho(x) = \HM(a)/a(x)$ and homogenizes to $\overline L(Q) = \HM(a) A_0:Q + \overline f$, where $\overline f(x) = \int f(x) \dd \rho(x)$.
\end{lemma}

For the nonlinear operator $F$, the homogenized operator is given by solving the cell problem, see~\cite{Evans1989}.
\begin{definition}[Solution of the cell problem]
  Given $F$ uniformly elliptic, for each $Q \in \Sd$, there is a unique value
  $\overline F(Q)$ and a periodic function $u^Q(y)$  which is a viscosity solution of the
  cell problem
  \begin{equation}\label{def:cell_problem}
    F(Q + D^2u^Q(y), y) = \overline F(Q).
  \end{equation}
\end{definition}

\begin{lemma}[Homogenization of linearized operator]\label{def:linear}
  Consider the nonlinear elliptic operator $F(Q,x)$, and suppose for a given
  $Q$, that $\grad_Q F(Q,x)$ exists for all $x$.
  The corresponding linearization at $Q$ is given by \eqref{LFaffine}.
  Let $\rho^Q$  be the corresponding unique invariant probability measure,  which is the solution of the adjoint equation
  \begin{equation}\label{adjoint-equation}
    D^2:(\grad_Q F(Q,y)\rho^Q(y)) = 0,
  \end{equation}
  interpreted in the weak sense.
  Then $\overline{L^Q}(Q)$, the homogenized linearized operator evaluated at $Q$,  is given by
  \begin{equation}\label{LbarQ}
    \overline{L^Q}(Q) = \int_\Td F(Q,y) \dd \rho^Q(y).
  \end{equation}
\end{lemma}

\begin{proof}
%	  The homogenized PDE operator is given by %
%  \[
%    \overline {L^Q}(M) = \overline {A^Q} : M + \overline {g^Q}
%  \]
%  where
%  \[
%    \overline {A^Q} = \int_\Td A^Q(x) \dd \rho^Q(x),
%    \qquad
%    \overline {g^Q} = \int_\Td (F(Q,x) - A^Q(x):Q ) \dd \rho^Q(x).
%  \]
  The invariant measure $\rho^Q$ solves \eqref{adjoint-equation}, see \cite{bensoussan2011asymptotic} or \cite{Froese2009}.  Apply \eqref{LFaffine} at $M=Q$ and then integrate against $\rho^Q$ to obtain the result.
\end{proof}

\section{Main Result}

\subsection{Generalized semiconcavity estimates on the operators}
Consider the uniformly elliptic operator $F(Q,x)$, where $Q \in \mathcal{S}^d$ and $x \in \mathbb{T}^d$.
We assume the following.
\begin{ass}[Quadratically dominated for $F(Q,x)$]\label{def:quaddom}
  Let $F$ be as in Definition~\ref{def:F}.
  Suppose for a given $Q$, that $\grad_Q F(Q,x)$ exists for all $x$.  Write $\| Q \|$ for the Frobenious norm of $Q$.
  We say that $F$ is
  \emph{quadratically dominated above} at $Q$ if there is a bounded function $C^+(Q,x): \mathbb{T}^d \to \R$ such that
  \begin{equation}\label{QDabove}
    F(M,x) - L^Q(M,x)   \leq  C^+(Q,x)\frac {\norm{M-Q}^2}2,
    \quad \text{ for all $(M,x) \in \mathcal{S}^d \times \mathbb{T}^d$}
  \end{equation}
  and similarly, $F$ is
  \emph{quadratically dominated below}  at $Q$ if there is a bounded function $C^-(Q,x): \mathbb{T}^d \to \R$ such that
  \begin{equation}\label{QDbelow}
    F(M,x) - L^Q(M,x)   \geq  C^-(Q,x)\frac {\norm{M-Q}^2}2,
    \quad \text{ for all $(M,x) \in \mathcal{S}^d \times \mathbb{T}^d$}
  \end{equation}
\end{ass}

\begin{remark}
  If $F$ is convex in $Q$, then $C^-(Q,y) = 0$. Similarly if $F$ is concave in
  $Q$, $C^+(Q,y)=0$.  More generally if $F$ is semi-concave, or semi-convex in $Q$, then we can set $C^{\pm}(Q,y) = C^\pm(y)$, to be a constant independent of $Q$.   However, we require the definition above for when the semi-concavity or semi-convexity conditions in $Q$ do not hold, as is the case for the Pucci-type operators defined below.
\end{remark}

\begin{example} Let $x\in \R$ and set
  $f(x) = \max\left\{ a x, b x \right\}$.  Since $f$ is convex, we can take $C^-(x) =0$ in \eqref{QDbelow}.
  We claim that for $x\not=0$, \eqref{QDabove} holds with
  \begin{equation}
    \label{eq:CplusEx1}
    C^+(x) = \frac{\abs{a-b}}{2 x},
  \end{equation}
  and this is the best constant.
  See \Cref{fig:CQ_max_ab}.
\end{example}

\begin{proof}[Derivation of \eqref{eq:CplusEx1}]
  Expand $f(x +y)$ about the point $x$, for $x\not=0$. To test \eqref{QDabove},
  replace the inequality with an equality to obtain a quadratic equation.  By
  requiring that there is only one root, we obtain an equation for the
  discriminant of the quadratic, which can be solved to obtain the result.
\end{proof}

\begin{figure}[t]
  \centering
  \includegraphics[width=0.5\textwidth]{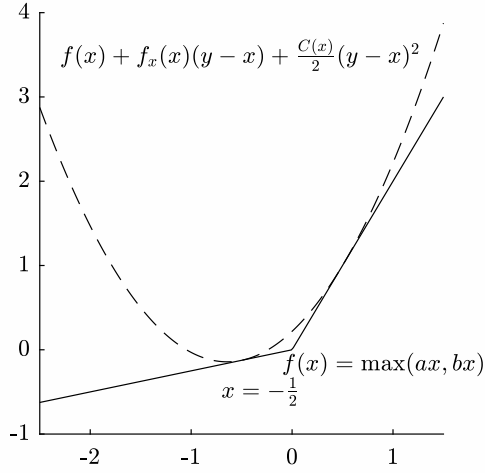}
  \caption{ For the simple example $f(x) = \max\left\{ a x, b
  x \right\}$, the semi-concavity constant is $C(x) = C^+(x) = \frac{|a-b|}{2 x}$.}
  \label{fig:CQ_max_ab}
\end{figure}

\subsection{Main Theorem}

\begin{theorem}  \label{th:main}
  Suppose $F$ satisfies Assumptions~\ref{def:quaddom} and
  $u^Q \in C^{2,\alpha}(\Td)$ is a classical solution.
  Let $\overline F(Q)$ be the homogenized operator at $Q$ and let  $u^Q$ be the corresponding solution of the cell problem given by~\eqref{def:cell_problem}.
  Let the homogenization of the linearization of the operator be given by \eqref{LbarQ} and let $\rho^Q(y)$ be the corresponding  invariant measure of the linearized problem \eqref{LFaffine}.
  Write
  \[
    \overline{ C^\pm}(Q) = \frac 1 2 \int C^\pm(Q,y) \| D^2u^Q(y) \|^2 \dd \rho^Q(y)
  \]
  Then
  \begin{equation}\label{eq:bounds}
    \overline{ C^-}(Q)
    \leq \overline F(Q) - \overline{L^Q}(Q)
    \leq \overline{ C^+}(Q)
  \end{equation}
\end{theorem}

\begin{remark}
  In the examples we consider below, $C^\pm(Q,y) \to 0$ as $\dist(|Q|,S) \to
  \infty$, for the singular set of the operator.  This gives control over the homogenization error for many values of $Q$.  Another term in the error is  $\|D^2u^Q\|$.
  In the homogeneous order one case, we have $u^Q = 0$ for $Q = 0$, so a
continuity argument suggests that we may have control of $\|D^2u^Q\|$ for small
values of $Q$. This is the case in one dimension in \cite{ObermanFinlay_LP}, where we obtain an
analytical formula  for $u^Q_{xx}$  through \eqref{separable}, which gives $\abs{u^Q_{xx}}
\leq C|Q|$.

The main theorem is a formal result in the sense that it 
 relies on the fact that $u^Q$ is a
  classical solution, which does not hold in general.   
If $F$ is convex (or concave), then by a
  famous theorem of Krylov and Evans
  \cite{krylov_boundedly_1984,evans_classical_1982}, or~\cite{caffarelli1995fully}, $u^Q  \in C^2(\Td)$. 
However, in general we are only guaranteed $u^Q \in C^{1,\alpha}(\Td)$ \cite{jensen_maximum_1988}.  
\end{remark}

\begin{proof}
  Subtract the linearization of $F$ at $Q$ evaluated at $Q + D^2u^Q(y)$ from the equation for the cell problem \eqref{def:cell_problem}, to obtain
  \begin{equation}    \label{eq:subtractlin}
    \overline F(Q) - L^Q(Q+D^2u^Q,y) = F(Q+D^2u^Q,y) - L^Q(Q+D^2u^Q,y).
  \end{equation}
  From Assumption~\eqref{def:quaddom},
  \begin{equation}
    \overline F(Q) - L^Q(Q+D^2u^Q,y) \leq C^+(Q,y) \frac{\| D^2 u^Q \|^2}{2}
    \label{eq:Fb_above}
  \end{equation}
  and
  \begin{equation}
    \overline F(Q) - L^Q(Q+D^2u^Q,y) \geq C^-(Q,y) \frac{\| D^2 u^Q \|^2}{2}.
    \label{eq:Fb_below}
  \end{equation}
  Now integrate \cref{eq:Fb_above,eq:Fb_below} against the invariant measure
  $\rho^Q$. This yields the upper and lower bounds \eqref{eq:bounds}, where we have used the fact that
  for all $\phi \in C^2(\Td)$,
  \begin{equation}
    \label{LbarAtPhi}
    \int_\Td {L^Q}(Q + D^2\phi,y) \dd\rho^Q(y)  = \int_\Td
    F(Q,y)\dd\rho^Q(y),
  \end{equation}
  which follows from  integration by parts, since $\rho^Q$  solves the adjoint equation~\eqref{adjoint-equation}.
\end{proof}

\subsection{Applications of the main result}
We give two applications of the main result.
In the first example, where the operator is separable,  we have an analytical formula for $\overline{L^Q}(Q)$.  In this case the estimates also simplify.  In the second, nonseparable example, we can find $\overline{L^Q}(Q)$ by solving a single linear homogenization problem, with coefficients given by the linearization \eqref{eq:LinearHP}.

\begin{corollary}
  \label{formula:sep}
  Consider the separable, purely second order operator
  \[
    F(Q, y) = a(y) F_0(Q)
  \]
  for $y \in \R^d$.  Suppose that $F_0$ is quadratically dominated with constants $C^-(Q)$ and $C^+(Q)$.  Then,
  \begin{equation}\label{eq:corbounds}
      \overline{ C^-}(Q)
%\\
    \leq
    \overline{F}(Q) - \HM(a) F_0(Q)
%  \\
    \leq
%    \HM(a)  C^+(Q)     \| D^2u^Q(y) \|^2
  \overline{ C^+}(Q)
  \end{equation}
%    \HM(a)  C^-(Q) \| D^2u^Q(y) \|^2
%  \end{multline*}
where
\[
    \overline{ C^\pm}(Q) = \frac 1 2 C^\pm(Q)\HM(a) \int_\Td  \| D^2u^Q(y) \|^2
    \frac{1}{a(y)} \dd y.
\]
\end{corollary}

\begin{proof}
  1. The formula for the linearization,
  \[
    \overline{L^Q}(Q) = \HM(a) F_0(Q)
  \]
  follows from the Linear Homogenization Formula (Lemma \ref{formula:HM}).

  2.  From linearization, we have that $\rho^Q(y) =
  \HM(a)/a(y)$.
  Using the definition, then the generalized semiconvexity/concavity constants for $F(Q,y)$ are given by
  \[
    C^+(Q,y) =  a(y)C^+(Q),
    \quad\text{ and }\quad
    C^-(Q,y) =  a(y)C^-(Q).
  \]
  Passing the constants and the invariant measure into Theorem~\ref{th:main}
  gives the bounds provided by \eqref{eq:corbounds}, since the coefficients $a(y)$ cancel.

\end{proof}

\begin{remark}
  In a companion paper \cite{ObermanFinlay_LP}, we show
that for convex operators in one dimension, 
  \begin{equation}\label{separable}
    \overline F(Q) =  \overline{L^Q}(Q) = \HM(a) F_0(Q).
  \end{equation}
\end{remark}

\begin{corollary}
  Consider the operator $H^{a_1,a_2}$ given by \eqref{eq:HP}.   Then
  \begin{multline*}
    \left| \overline{H^{a_1,a_2}}(Q) - \overline {L^Q}(Q) \right|
    \\    \leq
    \frac{1}{2\abs{\lambda_{\min}(Q)-\lambda_{\max}(Q)}} \int \abs{a_1(y)-a_2(y)} \, \| D^2u^Q(y) \|^2 \, d\rho^Q(y)
  \end{multline*}
\end{corollary}

\begin{proof}
  We apply Theorem~\ref{th:main} to $H^{a_1,a_2}$ given by \eqref{eq:HP}.
  The linearization is given by \eqref{eq:LinearHP}.  The invariant measure of the linear problem is given by the solution of \eqref{adjoint-equation} and the homogenized linear operator is given by \eqref{LbarQ} from Lemma~\ref{def:linear}.

  The main step is to work out the generalized semi-concavity constants.   We claim.
  \begin{equation}
    C^+(Q,x) =  \frac{ (a_2(x)-a_1(x))^+}{\abs{\lambda_{\min}-\lambda_{\max} }},
    \quad\text{ and }\quad
    C^-(Q,x) =  \frac{(a_1(x)-a_2(x))^-}{\abs{\lambda_{\min}-\lambda_{\max}} }.
    \label{HSCconstant}
  \end{equation}

  To prove this we proceed in steps.

  1. First, take $q\in \R^2$ and set
  $f(q) = \max(q_1, q_2)$.  Then $L^q(y) = \grad f(q)\cdot y$ away from the
  singular set $q_1=q_2$, since the
  function is homogeneous of order one. The constant $C^-(q) = 0$, since $f$ is
  convex.   We claim the optimal choice for $C^+(q)$ is given by
  \[
    C^+(q) = \frac{1}{\abs{q_2 - q_1}}
  \]
  for $q_1\not= q_2$.  To see this, we require
  \[
    \max(y_1,y_2) \leq \grad f(q)\cdot y + \frac{C^+(q)}{2} \abs{y-q}^2.
  \]
  It is easily verified that the extremal case for the inequality   occurs when $(y_1,y_2) = (q_2, q_1)$, which leads to the condition
  \[
    \abs{q_1-q_2} \leq {C^+(q)} \abs{q_1-q_2}^2.
  \]
  giving the result.

  2. Let 	$f(q_1, q_2) =  a_1\min(q_1, q_2) +  a_2\max(q_1, q_2)$.
  Rewrite $f(q_1, q_2) =  a_1(q_1 + q_2) +  (a_2-a_1) \max(q_1, q_2)$.  We can always subtract an affine function when computing the constants.  So the constants for $f$ are the same as the constants for $(a_2-a_1) \max(q_1, q_2)$.  In this case, using the result of step 1, we obtain
  \[
    C^+(x) =  \frac{(a_2-a_1)^+}{\abs{q_1-q_2}},
    \qquad
    C^-(x) = \frac{(a_1-a_2)^-}{\abs{q_1-q_2}}
  \]

  3.  Next consider for the two by two matrix $Q$, $h(Q,x) = a_1(x) \min(q_{11}, q_{22}) + a_2(x) \max(q_{11}, q_{22})$.  Then the previous step shows that the constants for $h$ are given by the previous ones (with $q_{11}$ replacing $q_1$ and $q_{22}$ replacing $q_2$.   Finally, since $H^{a_1, a_2}$ depends only on the eigenvalues of $Q$, without loss of generality, we can choose a coordinate system where $Q$ is diagonal when computing the generalized semiconcavity constants.   It remains to show that the generalized semi-concavity condition holds for a matrix, $M$.  If $M$ is diagonal the condition holds.  But if $M$ is not diagonal, then the change in the norm $\| M - Q\|^2$  can be controlled by a constant, or absorbed into the definition of the norm.
\end{proof}

%%%%%%%%%%%%

\section{Computational Setting}

For our numerical experiments, we consider a wider class of separable and non-separable operators.

\subsection{PDE Operators}
\begin{definition}[Pucci-type operators]\label{defnPucci}
  For $\delta > 0$ and given functions $0 < \delta \leq  a(y)\leq A(y)$.  Write $b(y) = A(y)-a(y)$.
  Also write $t^+ = \max(t,0)$.
  Define, for $d=2$, the standard Pucci maximal operator, the Pucci-type
  operator, the smoothed Pucci-type operator, and a Monge-Ampere type operator
  respectively as
\begin{align}
     \label{Pucci_2d}
    P^{A,a}(Q,y) &= a(y) \tr Q + b(y)\left ( \lambda_{\min}^+(Q) + \lambda_{\max}^+(Q) \right )
    \\
    \label{eq:OurPucci}
    F^{A,a}(Q,y) &= a(y) \tr Q + b(y) \lambda_{\max}^+(Q).
    \\
    F_k^{A,a}(Q,y) &= a(y) \tr Q + b(y)\mathcal S_k(\lambda_{\min}(Q),\lambda_{\max}(Q),0)
    \label{eq:smoothP}
    \\
    M(Q,y) &=  a(y)\left( \tr(Q) + \lambda_{\min}^+(Q) \lambda^+_{\max}(Q)
    \right).
  \end{align}
  Here, for $x \in \R^m$, $S_k(x)$ is the smoothed maximum function,
  \begin{equation}
    \mathcal S_k(x) = \frac{\sum_{i=1}^m
  x_i \exp(k x_i)}{\sum_{i=1}^m \exp(k x_i)}.
  \label{eq:smoothmax}
  \end{equation}
  The function $S_k$ goes to the max as $k\to \infty$, and to the average as $k\to 0$.
\end{definition}

\begin{definition}[Periodic checkerboard, stripes, and random checkerboard coefficients]
  Define
  \begin{equation}a_0(y) = \begin{cases}1,~ y\in B\\r,~y\in
      W,\end{cases}\label{eq:a0}\end{equation} with $r>1$. The
  sets $B$ and $W$ are either black and white squares of a checkerboard;
  alternating black and white stripes of equal width; or a `random' checkerboard,
  with black and white squares distributed with equal probability, uniformly.
\end{definition}

\begin{remark}[Representation and visualization of the operators]
  The definition above agrees with the usual definition of the Pucci operator,
  \begin{equation}
    P^{A,a}(Q,y) = \sup \{ M : Q \mid a(y) I \ll M \ll A(y) I  \}.
    \label{eq:RealPucci}
  \end{equation}
%We can write
%\[
%P^{A,a}(Q,y) %= P_q(q(Q),y)
%= \max \{ d \cdot q(Q) \mid
%a(y) \leq d_i \leq A(y), \quad \text{ for each } i  \}
%\]
  We can also rewrite
  \[
    F^{A,a}(Q,y) =
    \begin{cases}
      a(y) \tr Q & \text{$Q$ negative definite}
      \\
      H^{A,a}(Q,x), & \text{otherwise}.
    \end{cases}
  \]
\end{remark}

% --- VISUALIZATION OF EXAMPLE OPERATORS ---
\begin{figure}[t]
  \centering
  \begin{subfigure}[b]{0.45\textwidth}
    \includegraphics[width=\textwidth]{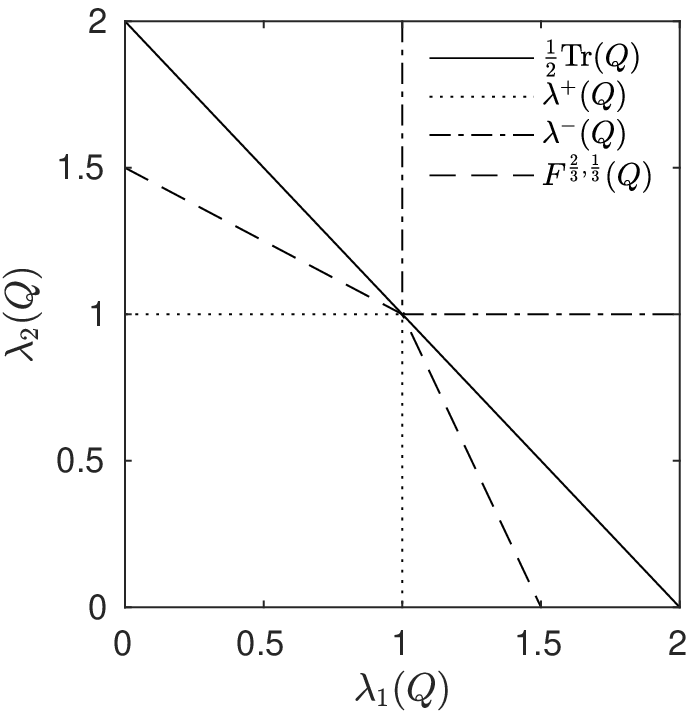}
    \caption{Example level sets}
    \label{fig:visualize}
  \end{subfigure}
  \hfill
  \begin{subfigure}[b]{0.475\textwidth}
    \includegraphics[width=\textwidth]{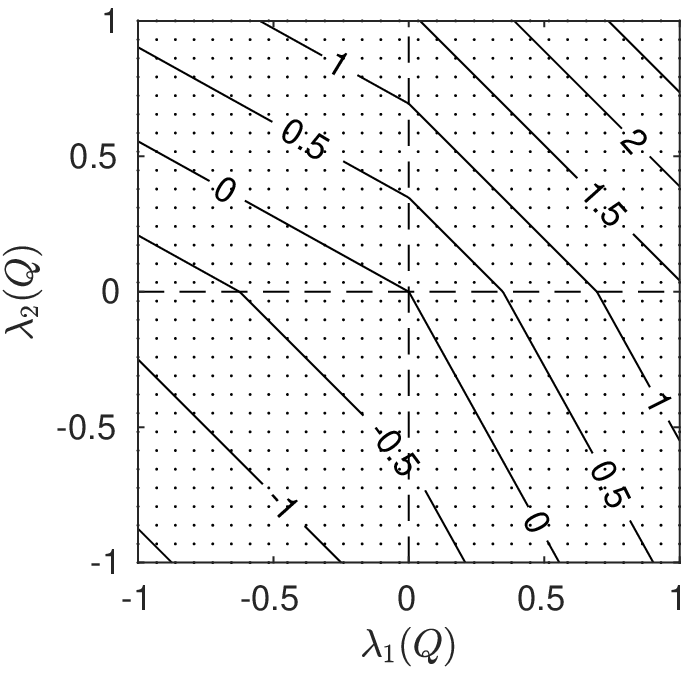}
    \caption{Pucci level sets}
    \label{fig:pucci-level-sets}
  \end{subfigure}
  \caption{Figure \ref{fig:visualize}: level set plot of several operators as
  function of the eigenvalues of $Q$.  Figure \ref{fig:pucci-level-sets}: Level
  sets of an example Pucci operator, $P^{\frac{5}{4},\frac{2}{3}}(Q)$. Points
indicate values of $Q$ that were homogenized.}
  \label{fig:pucci-figs}
\end{figure}

\begin{example} For the Pucci operator $P^{A,a}(Q,x)$ given by \eqref{Pucci_2d}, by convexity, $C^-(Q,x) = 0$ and
  \begin{equation}
    C^+(Q,x) =
    \begin{cases}
      \max\left\{ \frac{b(x)}{\tr(Q)},\frac{A(x)}{2\lambda_{\min}} \right\},
      &\text{ if } \lambda_{\min}, \lambda_{\max} > 0 \\
      \frac{b(x)}{2 \min(\abs{\lambda_{\min}},\abs{\lambda_{\max}} )}, &\text{ otherwise}.
    \end{cases}
    \label{eq:CQ_pucci}
  \end{equation}
\end{example}

\subsection{Numerical Method details}

In order to compute the errors, as a function of $Q$, we used a grid in the $\lambda_1-\lambda_2$ plane, and computed the linear and nonlinear
homogenization at the grid points.  Typical values can be see in
Figure~\ref{fig:pucci-level-sets}, where black points indicate the grid values
of $Q$ tested.

We remark on the numerical methods used throughout.
To compute $\overline{F} (Q)$ directly, we discretized with finite
differences and solved the parabolic equation $u_t + F(Q+D^2u,y)$ using an explicit
Euler method, to iteratively compute a steady state solution.   We discretized
using a convergent monotone scheme \cite{ObermanSINUM} and also using standard finite
differences.  The accuracy of the monotone scheme was less than the standard
finite differences, so we implemented a filtered scheme
\cite{froese2013convergent}.  In practice, the filtered scheme always selected
the accurate scheme, so in this instance, perhaps because the solutions are
$C^2$ and periodic, standard finite differences appear to converge.

For all the computations, to avoid trivial solutions, we solved with a right hand side function equal to a constant, and then subtracted the same constant from $\overline F(Q)$.

The computational domain was the torus $[0,1]^2$, divided into $20\times20$ equal
squares, each with 16 grid points per square.

% --------------------------------------------

\begin{remark}[Comparison with \cite{caffarelli2008numerical}]
  The problem of homogenizing $a_0(y) F^{A,a}(Q)$, was considered in \cite{caffarelli2008numerical}.
  In their case, the spatial coefficient $a_0(y)$ varies periodically and smoothly between 2 and 3,
  and their homogenized value for $\overline a_0$ was
  2.5 (which was the average of the coefficient $a_0(x,y) = \cos(\pi x)\cos(\pi y)$).  Our
  results using these coefficients was $\overline a_0 = 2.486$,  which is very close to the average.   However with coefficients which are more spread out, we obtain values far from the average.
\end{remark}

% --- FIG: PUCCI OPERATOR and CQ ---
\begin{figure}[t]
%  \centering
%  \begin{subfigure}[b]{0.48\textwidth}
%    \includegraphics[width=\textwidth]{real_pucci_separable-chkbd-log-error_r=2}
%    \caption{Error}
%    \label{fig:real-pucci-error}
%  \end{subfigure}
  \begin{subfigure}[b]{0.52\textwidth}
    \includegraphics[width=\textwidth]{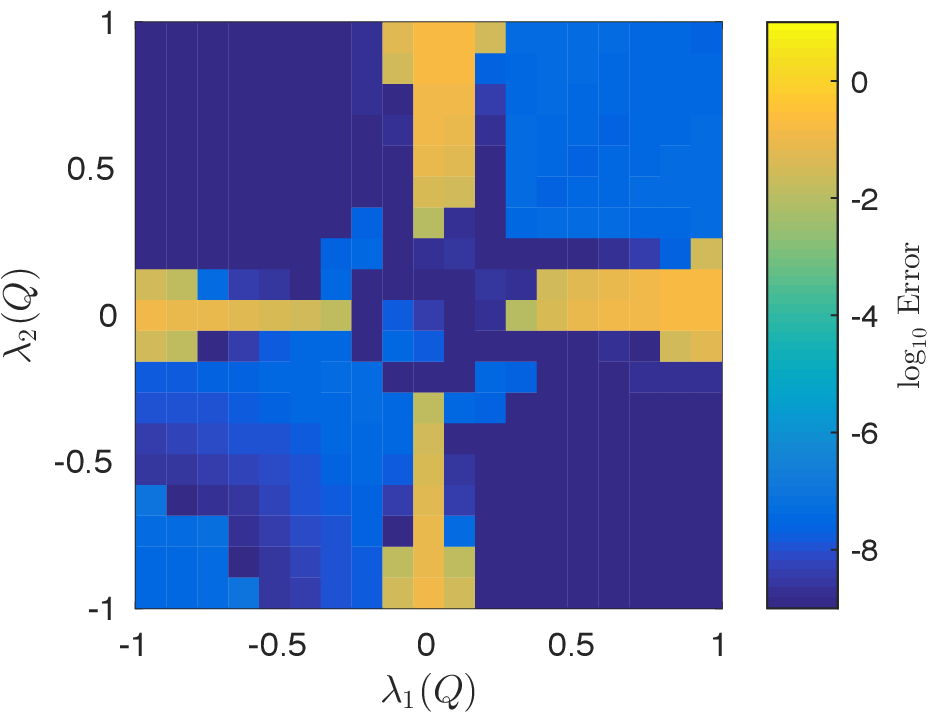}
    \caption{Error}
    \label{fig:real-pucci-error2}
  \end{subfigure}
  \hfill
  \begin{subfigure}[b]{0.45\textwidth}
    \includegraphics[width=\textwidth]{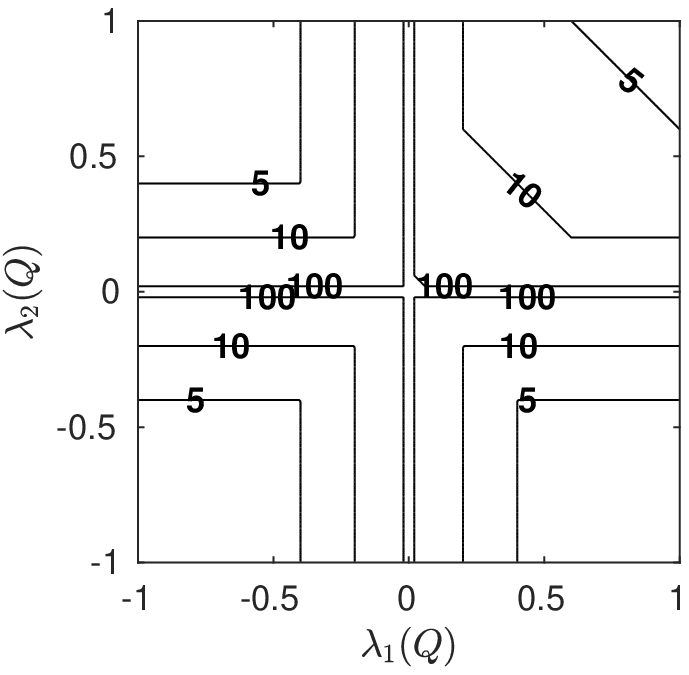}
    \caption{$\max_y C^+(Q,y)$}
    \label{fig:CQ}
  \end{subfigure}
  \caption{Homogenization of a separable Pucci example operator, $a(y)
    P^{3,1}$, on a periodic checker board, with coefficients of 1 or 2 ($r=2$).   \ref{fig:real-pucci-error2}: Error $\overline F(Q) - \overline L^Q(Q)$.
    %\ref{fig:real-pucci-error}: level set plot.
    Figure \ref{fig:CQ}:
    An upper bound of the semi-concavity constant $C^+(Q,y)$.
    The error is \num{1e-6} or less in the blue part of the domain.  In the yellow region it goes from 0.01 up to 0.15.
    The regions where the error is small coincide with smaller values of the semi-concavity constant.
  }
  \label{fig:real-pucci-chkbd-r=2}
\end{figure}

\section{Numerical results}\label{sec:Numerics}

\subsection{Numerical Results: separable operators}\label{sec:numSepPucci}
Here we check the homogenization error of the bound for separable operators in
two dimensions, from Corollory~\ref{formula:sep}.  We are in the convex case, so the lower bound is zero.

We performed numerical simulations on four operators, see Definition~\ref{defnPucci}.
\begin{itemize}
  \item $a_0(y) P^{3,1}(Q)$
  \item $a_0(y) F^{3,1}(Q)$
  \item $a_0(y) F^{3,1}_k(Q)$, with  $k=10$ and $k=0.1$
  \item $a_0(y) M(Q)$
\end{itemize}

In \Cref{fig:real-pucci-chkbd-r=2} we compare the error $\overline F(Q) -
\overline L^Q(Q)$ for a separable Pucci operator on a checkerboard, we also
illustrate the constant $C^+(Q,y)$.  In this case we have an analytical formula
for $\overline L^Q(Q)$).  This figure illustrates the Main Theorem: when the
constant is large the error from the linearization is high.  The error
is less than \num{1e-6} outside of a small region about the axis, and on the
order of $0.1$ near the axis.

In \Cref{fig:smoothmax}, we show how the error $\overline F(Q) - \overline
L^Q(Q)$ decreases as the operator becomes smoother.    The operator with the
smallest maximum curvature (\Cref{fig:psmooth10}) exhibits the smallest error.
As the operator becomes less smooth, the error increases.  For the smoothest
operator the global error is at most one percent (in the range of values shown
in the figure).  For the two sharper operators, there is still very high
accuracy away from the highest curvature regions.
We see that error of the
smooth operator, $F^{3,1}_{10}(Q)$, is slightly smaller than the non smooth
operator's error near the line $\lambda_1=\lambda_2$ (this is where the non
smooth operator is not differentiable). As the smoothing constant $k\rightarrow
0$, the error of the linearized homogenization decreases. For example, setting
$k=0.1$, as in \Cref{fig:psmooth10}, results in an error on the order of $0.01$.
In all cases, the error is near zero in a large part of the domain.  It
concentrates near the positive diagonal, where it is .1 to .4 for the non-smooth
operator, and similar for the operator with a small smoothing parameter.  A
larger smoothing parameter sends the error in a similar region to the range .002
to 0.01.
A small amount of smoothing has a small effect on the error.  More smoothing
leads to errors going from .1 to .002 in a similar part of the domain.

% --- COMPARISON WITH SMOOTH OPERATOR ---
\begin{figure}[t]
  \centering
  \begin{subfigure}[b]{0.325\textwidth}
    \includegraphics[width=\textwidth]{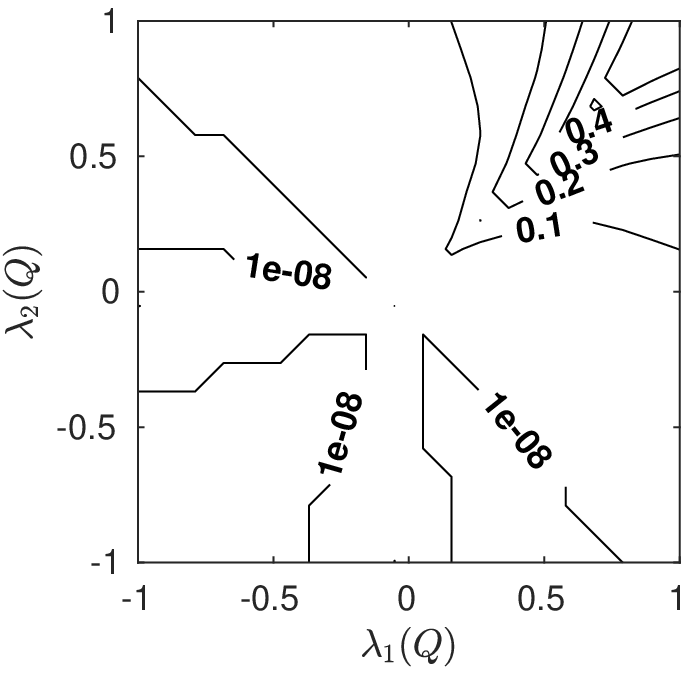}
    \caption{$a_0(y) F^{3,1}(Q),\, r=2$}
    \label{fig:psep}
  \end{subfigure}
  \hfill
  \begin{subfigure}[b]{0.325\textwidth}
    \includegraphics[width=\textwidth]{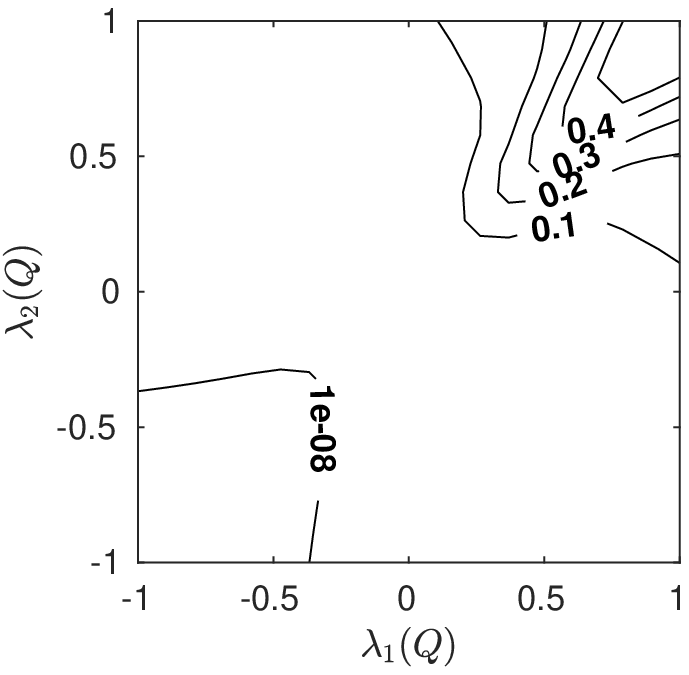}
    \caption{$a_0(y) F^{3,1}_{10}(Q),\, r=2$}
    \label{fig:psmooth100}
  \end{subfigure}
  \hfill
  \begin{subfigure}[b]{0.325\textwidth}
    \includegraphics[width=\textwidth]{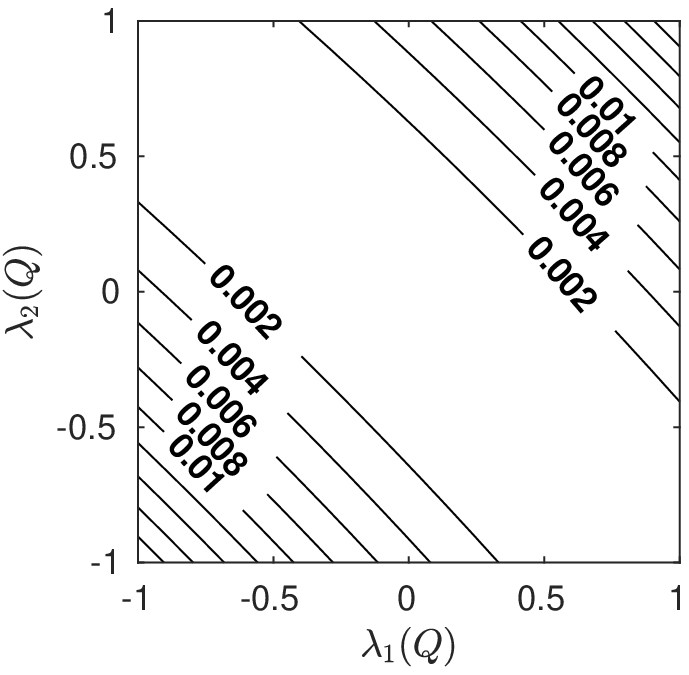}
    \caption{$a_0(y) F^{3,1}_{0.1}(Q),\, r=2$}
    \label{fig:psmooth10}
  \end{subfigure}
  \caption{Homogenization error for a smoothed Pucci type operator. The coefficients $a(y)$ are on a checker board with $r=2$ (i.e. $a$ = 1 or 2).
    The operators are defined in
    \Cref{sec:Numerics}. Figure \ref{fig:psep}: error on a Pucci like
    operator. \Cref{fig:psmooth100,fig:psmooth10}: error on a smoothed Pucci
  like operator.   }
  \label{fig:smoothmax}
\end{figure}

% --- FAKE PUCCI ON STRIPES ---
\begin{figure}[t]
  \centering
  \begin{subfigure}[b]{0.48\textwidth}
    \includegraphics[width=\textwidth]{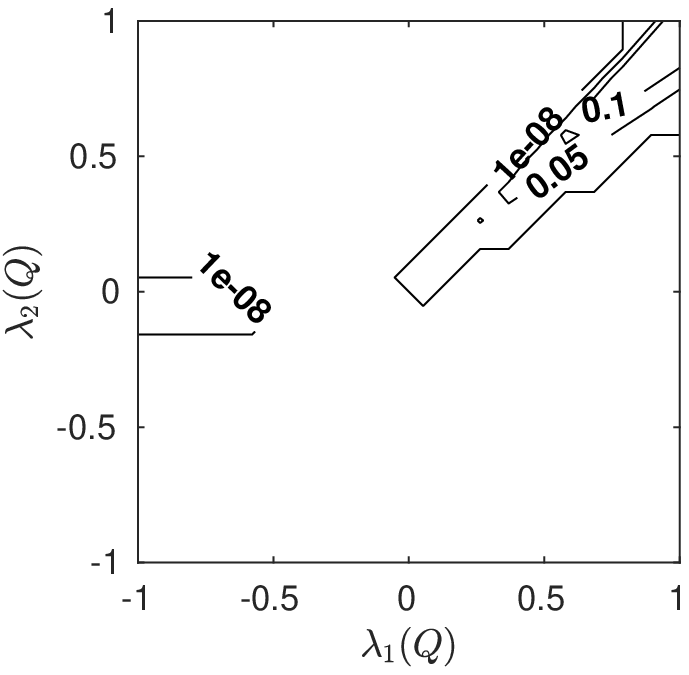}
    \caption{$r=1.2$}
    \label{fig:fake-pucci-stripes-r=1.2}
  \end{subfigure}
  \hfill
  \begin{subfigure}[b]{0.48\textwidth}
    \includegraphics[width=\textwidth]{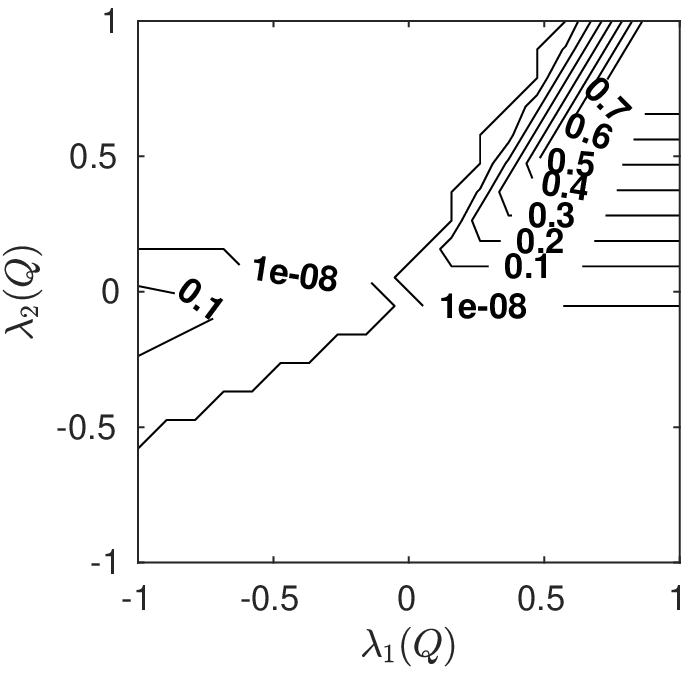}
    \caption{$r=2$}
    \label{fig:fake-pucci-stripes-r=2}
  \end{subfigure}
  \caption{Error for $a_0(y) F^{3,1}$ on stripes, with different ratios $r$.
}
\label{fig:fake-pucci-stripes}
\end{figure}

% --- MA + LAPLACE ON CHECKERBOARD
\begin{figure}[t]
  \centering
  \begin{subfigure}[b]{0.48\textwidth}
    \includegraphics[width=\textwidth]{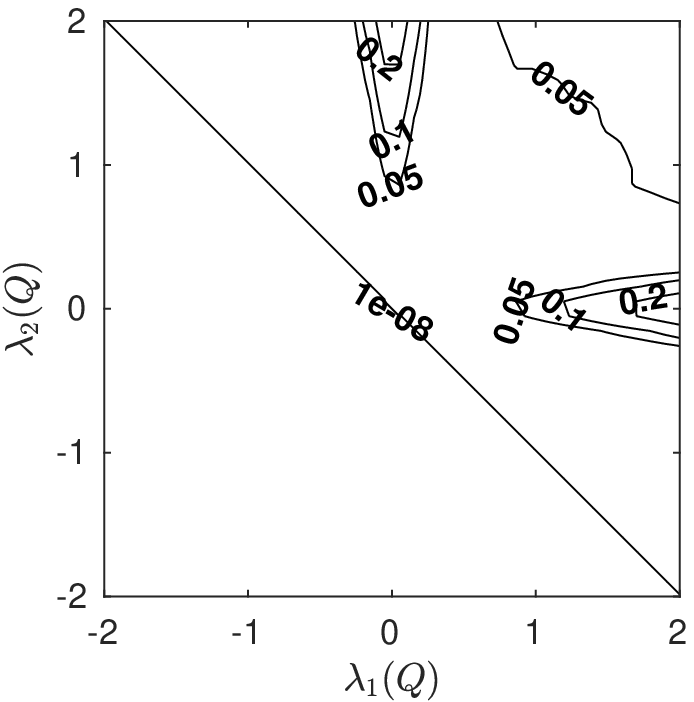}
    \caption{Checkerboard}
    \label{fig:ma-chkbd}
  \end{subfigure}
  \hfill
  \begin{subfigure}[b]{0.48\textwidth}
    \includegraphics[width=\textwidth]{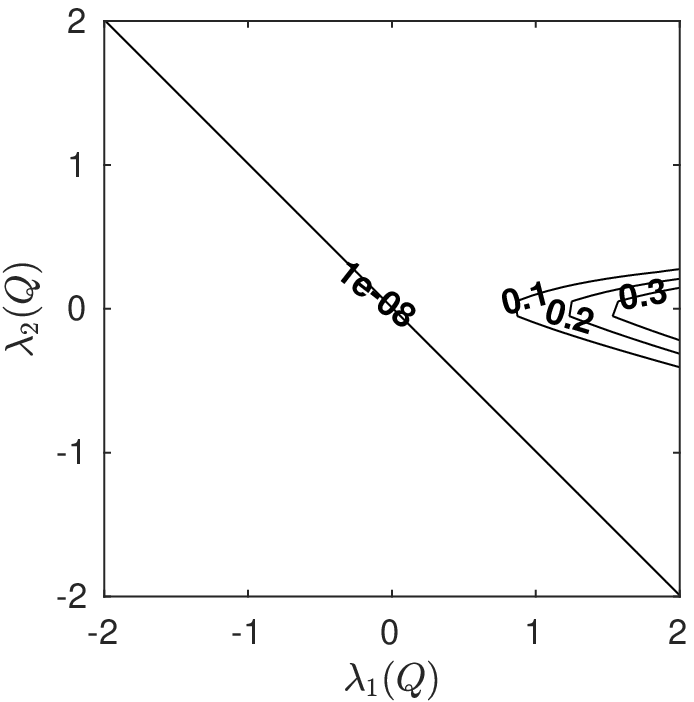}
    \caption{Stripes}
    \label{fig:ma-stripes}
  \end{subfigure}
  \caption{Error for $M(Q,y)$ with $r=2$, on a periodic checkerboard and
  on stripes.}
  \label{fig:ma}
\end{figure}

Figure \ref{fig:fake-pucci-stripes} presents the error for $a_0(y) F^{3,1}(Q)$ on
stripes. On stripes, the regions with large error are much smaller than the
operators on checkerboard. We hypothesize that this is because stripes have a
smoothing effect. The location where the large error is located
depends on the interplay between the operator and the direction of the stripes.
Given that in this example the homogenized operator is $\mathcal O(1)$, the
error here is particularly large.
In a companion paper, we will derive a closer lower bound for
$\overline{F^{A,a}}(Q)$, using the optimal invariant measure of the nonlinear
operator.

Figure \ref{fig:ma} shows error for $a_0(y) (\tr(Q) + \text{MA}(Q))$ on both
stripes and a periodic checkerboard. For the Monge Ampere type operator on
checkerboard, error is on the order of \num{1e-2} in the first quadrant, where
the curvature is bounded. Elsewhere the error is negligible.

As $r$ (the scaling coefficient of $a_0(y)$)  grows, so does the error. As expected, for the two Pucci
type operators on checkerboard, away from the regions where the curvature is
unbounded, the error is negligible: this is where the operators are linear.
Although we do not show it, for all figures, the error profile on the random
checkerboard is nearly identical to the periodic checkerboard.

\subsection{Numerical Results: non-separable operators}\label{sec:iso}
Now we consider nonseparable coefficients for $F^{A,a}(Q,y)$, refer to
Definition \ref{defnPucci}.

For both periodic and random checkerboard coefficients, the numerically computed
values of $\overline {F^{A,a}}(Q)$ depend only on the eigenvalues of $Q$, not on the
eigenvectors.  In addition,  $\overline{F^{A,a}}(Q)$ is homogeneous order one.  So the
entire function $\overline{F^{A,a}}(Q)$ is determined by the 1-level set of
$\overline{F^{A,a}}(Q)$ for
diagonal matrices $Q$.

We write
\begin{equation}\label{formula:guess}
  \overline{L^Q}(Q) = \overline A \lambda^+(Q) + \overline a \lambda^-(Q)
\end{equation}
where the coefficients are obtained by numerical homogenization of the linearized operator
\eqref{eq:LinearHP} when $Q$ had at least one positive eigenvalue.  (In the negative definite case the operator is linear and the error was insignificant).

We found that error was within 5\% for a range of values of $A$ and $a$ with coefficients which vary by
a factor of 10.

In Figure \ref{fig:nonseparable} we show the error on a periodic checkerboard, with
\[
  F^{A,a}(Q,y) = \begin{cases}
    \tr(Q),\, y\in B\\
    F^{4,1},\,y\in W.
  \end{cases}
\]
The error is on the order of \num{1e-1} near the line $\lambda^+=\lambda^-$ in
the first quadrant; on the order of \num{1e-2} in the second and fourth
quadrants; and negligible otherwise.
In Figure~\ref{fig:nonseparable} we plot the error  against the
numerically homogenized value for an the nonconvex operator alternating between
$F^{2,1}$ and $F^{1,\frac{4}{3}}$ on a periodic checkerboard.

% --- NON SEPARABLE PUCCI OPERATORS ON CHECKERBOARD ---
\begin{figure}[t]
  \centering
  \begin{subfigure}[b]{0.48\textwidth}
    \includegraphics[width=\textwidth]{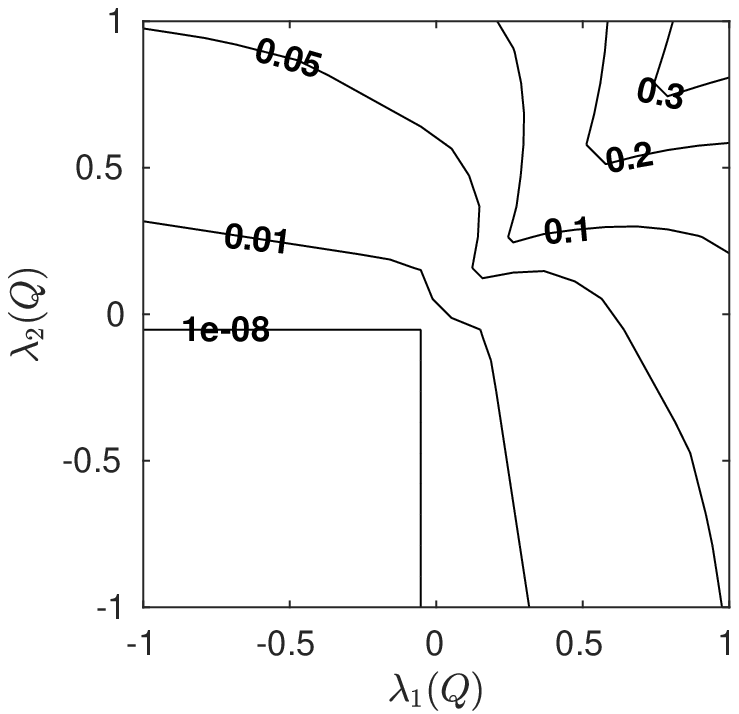}
  \end{subfigure}
  \begin{subfigure}[b]{0.48\textwidth}
    \includegraphics[width=\textwidth]{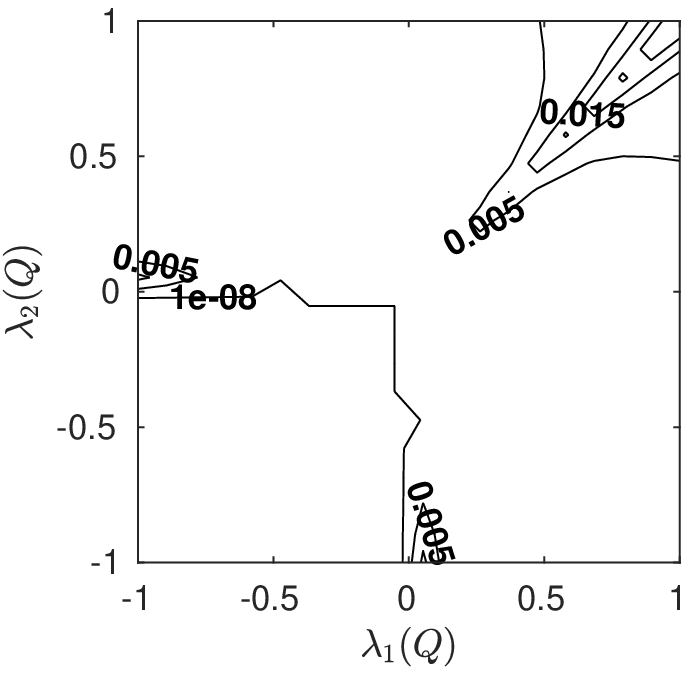}
  \end{subfigure}
  \caption{Error for the non-separable operator on a periodic checkerboard. Left: alternating between $F^{1,1}$ and
  $F^{4,1}$.  Right: alternating between $F^{2,1}$ and
  $F^{1,\frac{4}{3}}$}
  \label{fig:nonseparable}
\end{figure}

\subsubsection{Further experiments} We let $A$ and $a$ each take two positive values in
periodic checkerboard pattern. In the second, we let $A$ and $a$ each take two
positive values in a random checkerboard, drawn randomly from a Bernoulli
trial with probability $p$.  We checked both when $p=\frac{1}{2}$ and other
values of $p$. When $p=\frac{1}{2}$ the homogenized operator on the random
checkerboard is identical to the homogenization on the periodic checkerboard.
Finally, we also checked the case when $A$ and $a$ are each drawn from a uniform
distribution with positive support. In all of these cases, the numerically
homogenized operator is (numerically) isotropic, homogeneous order one, and
agrees closely with $\overline F$ in the approximate formula \eqref{formula:guess}.

\section{Conclusions}
We studied the error between the homogenization of the linearized operator and
the fully nonlinear homogenization.  We obtained upper and lower bounds on the
error in terms of the generalized semiconvavity constants of the operator.

We also performed numerical calculations.  For the class of operators we
studied, linearization was very accurate for a wide range of values of $Q$, with
negligible error in some cases.  The numerically computed errors were small, and
concentrated around regions of high curvature in $Q$ of the operator $F(Q,x)$.
Errors grew with the degree of nonlinearity and with the range of the
coefficients.

The numerical results are consistent with the bounds, although in some cases the error was smaller than was predicted by the bounds.

%
%\begin{remark}
%The  operator $H^{A,a}$ is not differentiable along the line $\lambda_1=\lambda_2$. The semiconcavity constants are inversely proportional to the distance to the singular set.  For the operator $P^{A,a}$, as similar remark holds:
%the constant, $C^+(Q,x)$, as a function of $Q$ is well approximated, up to constants in $x$, by  the inverse of the distance to the axes.
%
%As a function of the eigenvalues, the  operator, $P^{A,a}(Q)$,  has level sets with corners on the axes,
%the operator $F^{A,a}(Q)$ has level sets with corners on the positive diagonal, and the negative semi-axes, and the operator $H^{A,a}(Q)$ has level sets with corners on the diagonal.   Elsewhere, the operators are linear in the eigenvalues of $Q$.  See Figure \ref{fig:pucci-figs} for a visualization of the level sets of these two operators.
%\end{remark}
%
%\begin{remark}
%	We highlight that the error for operators on checkerboards (periodic or random) is significant only where
%curvature is large or unbounded.  For the operator $F^{A,a}(Q)$ this occurs
%along the line $\lambda^+=\lambda^-$ when $\lambda^+\geq\lambda^-\geq 0$. For
%the operator $P^{A,a}(Q)$ this occurs when one of the eigenvalues is
%zero.
%\end{remark}

\bibliographystyle{alpha}
\bibliography{HomogPucci}
\end{document}